\documentclass[11pt]{article}
\usepackage{amsfonts}
\usepackage{mathrsfs}
\usepackage{amsmath}
\usepackage{amssymb}
\usepackage{arydshln}
\usepackage{graphicx}
\usepackage{epsfig}
\usepackage{epic}
\usepackage{cite}
\usepackage{setspace}
\usepackage{amsthm,latexsym}
\usepackage{float}
\usepackage[utf8]{inputenc}

\renewcommand{\paragraph}{\roman{paragraph}}
\newtheorem{theorem}{ \bf Theorem}[section]
\newtheorem{lemma}[theorem]{ \bf Lemma}

\begin{document}

\title{\sf Multiplicity of Laplacian eigenvalue 1 of a graph}
\author{
	Yuhao Zhou,
    Fenglei Tian \thanks{Corresponding author. E-mail address: tflqsd@qfnu.edu.cn.  Supported by ``the Natural Science Foundation of Shandong Province (No. ZR2024MA032), the National Natural Science Foundation of China (No. 12101354) and the Youth Innovation Team Project of Shandong Province Universities (No. 2023KJ353)''.}
~~\\
\noindent{\small\it \ School of Management, Qufu Normal University, Rizhao, China.}\\
   }
\date{}
\maketitle

\noindent {\bf Abstract:} \ Let $G$ be a graph with $p(G)$ pendant vertices and $q(G)$ quasi-pendant vertices. Denote by $m_{L(G)}(\lambda)$ the multiplicity of $\lambda$ as a Laplacian eigenvalue of $G$. A graph $G$ is called reduced, if $p(G)=q(G)$. It is known that deleting a pendant path $P_3$ from a graph $G$  cannot change $m_{L(G)}(1)$. By the reduction operation for a graph (defined by Tian and Wong, 2026), we could turn to the reduced graphs with each quasi-pendant vertex of degree 2 to investigate $m_{L(G)}(1)$. Then let $T$ be a reduced tree on $n(\geq 7)$ vertices with each quasi-pendant vertex of degree 2 and without pendant path $P_3$. We first prove that
\begin{equation*}
    m_{L(T)}(1)\leq \frac{n-5}{6}
\end{equation*}
and the extremal trees attaining the upper bound are determined completely.
In addition, let $G$ be an arbitrary connected reduced graph with order $n\geq 6$ and size $m$. Denote by $c=m-n+1$ the first Betti number of $G$, then we obtain
\begin{equation*}
    m_{L(G)}(1)\leq c+\frac{n-2}{4},
\end{equation*}
and the extremal graphs attaining the upper bound are characterized completely.


\section{Introduction}

\quad
In this paper, we focus on simple, undirected and connected graphs. Assume that $G$ is a graph with the vertex set $V(G)$ and the edge set $E(G)$, and let the vertex $u\in V(G)$. Denote by $d(u)$ the degree of a vertex $u$, namely the number of vertices adjacent to $u$. A vertex $u$ is said to be a pendant vertex if $d(u)=1$, and its neighbor is called quasi-pendant vertex. If $u$ and $v$ are adjacent, then we write $u\thicksim v$ for brevity. The diameter of $G$ is denoted by $diam(G)$. The first Betti number of a connected graph $G$ is $c=|E(G)|-|V(G)|+1$, which is also called the cyclomatic number of $G$. Denote by $G-e$ a spanning subgraph of graph $T$ obtained by removing an edge $e$ from $G$. Let $p(G)$ (resp., $q(G)$) be the number of the pendant vertices (resp., quasi-pendant vertices) of  $G$. Denote by $|G|$ the order of $G$. For $S\subset V(G)$, let $G-S$ be an induced subgraph of $G$ obtained by deleting the vertices of $S$ with the incident edges. For brevity, we sometimes write $G-G_1$ instead of $G-V(G_1)$, where $G_1$ is an induced subgraph of $G$. Let $G-e$ be a spanning subgraph of $G$ obtained by removing an edge $e$.
We say that $G$ is reduced if $p(G)=q(G)$. We can obtain a reduced graph from a given graph $G$ by removing certain pendant vertices until every quasi-pendant vertex is adjacent to exactly one pendant vertex. Let $m_{M}(\lambda)$ be the multiplicity of the eigenvalue $\lambda$ of the matrix $M$. The Laplacian matrix of a graph $G$ is written as $L(G)=D(G)-A(G)$, where $D(G)$ is the degree diagonal matrix of $G$ and $A(G)$ is the adjacency matrix of $G$.

Much work has been done on the multiplicity of the Laplacian eigenvalues of a graph (see \cite{Akbari,Andrade,Barik,Du,Faria,Guo,Tian1,Wen} for instance). Here we focus on the multiplicity of Laplacian eigenvalue $1$ of graphs.
By Faria's result \cite{Faria} in 1985, it follows that $m_{L(G)}(1)\geq p(G)-q(G)$ for a graph $G$, and the upper bound is achieved when each vertex with degree greater than $1$ is a quasi-pendant vertex (see \cite{Andrade}).
Grone, Merris and Sunder \cite{Grone} proved that if $\lambda >1$ is an integer and $\lambda$ is a Laplacian eigenvalue of a tree $T$ of order $n$, then $\lambda \mid n$ and $m_{L(T)}(\lambda)=1$. Moreover, they showed that $m_{L(T)}(\lambda)\leq p(T)-1$ for any Laplacian eigenvalue $\lambda$ of a tree $T$.
All the trees attaining the upper bound $p(T)-1$ were determined by Yang and Wang \cite{Yang} and Wong $et\ al.$ \cite{WongZX}, independently.
Guo, Feng and Zhang \cite{Guo} investigated the affection of adding an edge between two disjoint graphs on the multiplicity of Laplacian eigenvalues of graphs.
Assume that $T$ is a tree formed by adding an edge between a vertex of a tree $T_1$ and a vertex of a tree $T_2$. Then, the relationship among $m_{L(T)}(1)$, $m_{L(T_1)}(1)$ and $m_{L(T_2)}(1)$ was investigated by Barik, Lal and Pati \cite{Barik}.
Wen and Huang \cite{Wen} considered the multiplicity of Laplacian eigenvalue $1$ of a unicyclic graph.
For a reduced tree $T$ of order $n(\geq 6)$, Tian $et\ al.$ \cite{Tian1} proved that $m_{L(T)}(1)\leq \frac{n-2}{4},$ and all trees attaining the upper bound were determined.

We call $P_{k}$ a pendant path of $G$, if $G$ can be obtained from a graph $H$ and a path $P_{k}$ by joining a pendant vertex of $P_{k}$ with an arbitrary vertex of $H$. For brevity, denote by $\mathcal{P}^{k}$ a pendant path $P_k$ in a graph $G$. As shown in \cite{Grone}, removing a $\mathcal{P}^3$ from a graph does not alter the multiplicity of Laplacian eigenvalue $1$. Let $\hat{G}$ be the reduced graph of a graph $G$. In \cite{Guo,Tian}, they found the relationship between the multiplicity of Laplacian eigenvalue $1$ of $\hat{G}$ and $G$:
$$m_{L(G)}(1)=p(G)-q(G)+m_{L(\hat{G})}(1).$$
Inspired by the results above, to investigate the multiplicity of Laplacian eigenvalue $1$ of a graph, one can just need to focus on the reduced graph without $\mathcal{P}^3$. Then Tian and Wong \cite{Tian} proved that, for a reduced tree $T$ on $n(\geq 6)$ vertices without $\mathcal{P}^3$, $m_{L(T)}(1)\leq \frac{n-6}{4}$ and all extremal trees were characterized.

To simplify the structure of a graph, the authors of \cite{Tian} defined a technique, called the reduction operation of a graph.

\noindent\textbf{Reduction Operation} \ Let $u$ be a pendant vertex of a graph $G$ and $u$ is adjacent to $v$. Let $N_G(v) \setminus \{u\} = \{v_1, v_2, \dots, v_s\}$ be the neighbors (except $u$) of $v$. For the graph $G-\{u, v\}$, joining each $v_i$ ($1 \leq i \leq s$) to a vertex of a path $P_2$, then we obtain the reduction graph of $G$, denoted by $G'$. See Fig. \ref{fig.2:eps} for an instance.
If one do reduction operation recursively to $G$ until each quasi-pendant vertex has degree 2, then we call $G'$ the \textit{final reduction graph} of $G$.

\begin{figure}[htbp]
	\centering
	\includegraphics[width=0.7\textwidth]{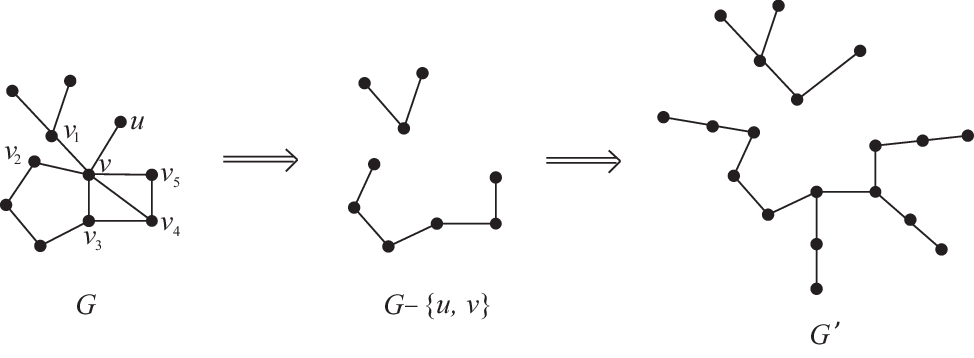}
	\caption{An example of doing reduction operation on a graph.}
	\label{fig.2:eps}
\end{figure}
According to the reduction operation, Tian $et\ al.$ \cite{Tian} showed that
$m_{L(G)}(1)= m_{L(G')}(1).$ Motivated by this result, to investigate $m_{L(G)}(1)$ of a graph $G$, we could consider its final reduction graph. Actually, we could study the multiplicity of Laplacian eigenvalue $1$ for the graph with each quasi-pendant vertex of degree 2. Thus, the first conclusion of this paper is obtained as below.

\begin{theorem}\label{thm:tree}\ Let $T$ be a reduced tree of order $n(\geq 7)$. If $T$ contains no $\mathcal{P}^3$ and each quasi-pendant vertex has degree $2$ in $T$, then
$$m_{L(T)}(1)\leq \frac{n-5}{6}$$
and the equality holds if and only if $T$ is isomorphic to the tree in Fig. $\ref{fig.1:eps}$.
\end{theorem}
\begin{figure}[htbp]
  \centering
  \includegraphics[width=0.25\textwidth]{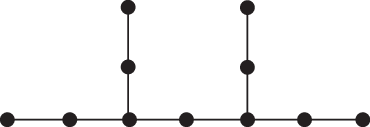}
  \caption{The extreme graph in Theorem \ref{thm:tree}.}
  \label{fig.1:eps}
\end{figure}

Apart from trees, the multiplicity of Laplacian eigenvalue $1$ of a unicyclic graph was also investigated. Wen and Huang \cite{Wen} proved that $m_{L(G)}(1)\leq n-2$ for a unicyclic graph $G$ of order $n$. Tian and Wong \cite{Tian} obtained $m_{L(G)}(1)\leq \frac{n}{4}$ for a reduced unicyclic graph $G$ without $\mathcal{P}^3$. To generalize these results, we further concentrate on $m_{L(G)}(1)$ for an arbitrary reduced graph $G$ and obtain the following conclusion.

\begin{theorem}\label{thm:circle}\ Let $G$ be a reduced graph on $n(\geq 6)$ vertices and $m$ edges.  Let $c=m-n+1$ be the first Betti number (or cyclomatic number) of $G$. Then
$$m_{L(T)}(1)\leq c+\frac{n-2}{4},$$
and equality holds if and only if $G$ is the cycle $C_6$ with $c=1$ or $G$ is a tree with $c=0$ satisfies the following two conditions:\\
(i) each vertex of degree greater than $2$ is a quasi-pendant vertex in $G$;\\
(ii) each of the final reduction components of $G$ is a path $P_6$.
\end{theorem}

Before proving Theorems \ref{thm:tree} and \ref{thm:circle} in Sections $3$ and $4$ respectively, we first introduce several lemmas in Section $2$.

\section{ Preliminaries }

\begin{lemma}\label{lem:edge}{\rm \cite{Grone}} \  Let $G-e$ be the graph obtained from a graph $G$ by deleting an edge $e$. Let $G-u$ be the graph obtained from a graph $G$ by deleting a pendant vertex $u$ and its incident edge. Then \\
 (i) \ $|m_{L(G)}(1)-m_{L(G-e)}(1)|\leq 1;$\\
 (ii) \ $|m_{L(G)}(1)-m_{L(G-u)}(1)|\leq 1.$
\end{lemma}

\begin{lemma}\label{lem:path}{\rm \cite{Grone}} \ Let $T$ be a tree obtained from a path $P_3$ and a tree $T_1$ by joining a vertex of $T_1$ to a pendant vertex of $P_3$. Then
$$m_{L(T)}(1)=m_{L(T_1)}(1).$$
\end{lemma}

\begin{lemma}\label{mainlemma}{\rm \cite{Tian}} \ Let $u$ be a pendant vertex of a graph $G$ and  let $v$ be the neighbor of $u$. Suppose that $d_v\geq 3$ and $w$, distinct from $u$, is adjacent $v$. Let $P_2=xy$ be a path on two vertices. Denote by $\Gamma$ the graph obtained from $G-e_{vw}$ and $P_2=xy$ by joining $w$ with  $y$. Then we have
$$m_{L(G)}(1)=m_{L(\Gamma)}(1).$$
\end{lemma}

The following result can be obtained by recursively using Lemma \ref{mainlemma}.
\begin{lemma}\label{lem:tree}{\rm \cite{Tian}} \ Let $u$ be a pendant vertex of a graph $G$ and $u\sim v$ with $d_{v}\geq 3$. Let $G'$ be the reduction graph of $G$ with respect to $u$ and $v$. Then
$$m_{L(G)}(1)=m_{L(G')}(1).$$
\end{lemma}

\begin{lemma}\label{lem:reduced}{\rm \cite{Tian1}} \ Let $T$ be a reduced tree with $n\geq 6$ vertices. Let $T_{i}'$ ($i=1,2,\cdots,k$) be all the final reduction components. Then
\begin{equation*}
	m_{L(T)}(1) \leq \frac{n-2}{4}
\end{equation*}
with equality if and only if the following two assertions hold: \\
(i) each vertex of degree greater than $2$ is a quasi-pendant vertex in $T$; \\
(ii) each $T_{i}'$ ($i=1,2,\cdots,k$) is a path $P_6$.
\end{lemma}

For a tree $T$ on $n(\geq 5)$ vertices, we call $T$ a {\it $P_2$-star} if there exists a central  vertex $u$ such that each component of $T-u$ is $P_2$. We call $T$ a {\it double $P_2$-star}, if it can be obtained from two $P_2$-stars by joining their central vertices with an edge.

\begin{lemma}\label{lem:star} \ Let $T$ be a $P_2$-star and $H$ be a double $P_2$-star. Then we have \\
(i){\rm \cite{Guo,Tian1}} \ $m_{L(T)}(1)=0$. \\
(ii){\rm \cite{Tian}} \ $m_{L(H)}(1)=0$.
\end{lemma}

\section{Proof of Theorem 1.1}
\noindent
{\bf Proof.} \ \ Assume that the set $\Omega$ consists of all reduced trees $T$ of order $n\geq 7$ without $\mathcal{P}^3$ and with each quasi-pendant vertex of degree $2$.
As shown in Fig. \ref{fig.3:eps}, when $n\leq 10$, $T$ is a $P_2$-star or a double $P_2$-star. Then we get $m_{L(T)}(1)=0$ by lemma \ref{lem:star}. When $n=11$, $T$ is a $P_2$-star or the extremal tree in Fig. \ref{fig.1:eps}. Thus, through the calculation and lemma \ref{lem:star}, we get $m_{L(T)}(1)=0$ or $m_{L(T)}(1)=1=\frac{11-5}{6}$.

\begin{figure}[htbp]
  \centering
  \includegraphics[width=0.65\textwidth]{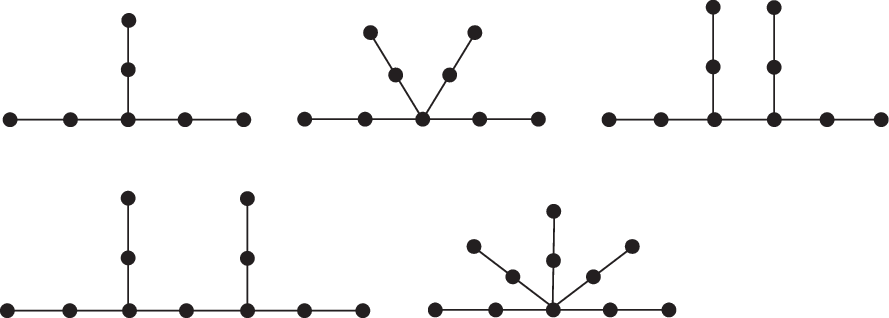}
  \caption{The trees belong to $\Omega$ with order less than $12$.}
  \label{fig.3:eps}
\end{figure}

Now, suppose that $n\geq 12$ and we prove $m_{L(T)}(1)<\frac{n-5}{6}$ by induction on order $n$. Let tree $T\in \Omega$. Since $n\geq 12$ and $p(T)=q(T)$, the diameter of $T$ is at least $4$, i.e., $diam(T)\geq 4$. Let $P_{k}=v_1 v_2 \cdots v_{k}(k\geq 5)$ be a diametrical path of $T$.
If $diam(T)=4$ or $diam(T)=5$, then $T$ is a $P_2$-star or a double $P_2$-star. Thus we obtain $m_{L(T)}(1)=0$ by Lemma \ref{lem:star}.

Suppose that $diam(T)=6$. In this case, let $T+u$ be the graph obtained from $T$ and a new vertex $u$ by joining $u$ with $v_4$. Using reduction operation to $T+u$ with respect to $u$ and $v_4$, we denote by $T_{i}' \ (i=1,2,\cdots, t)$ all the components of the reduction graph of $T+u$. Noting that $diam(T)=6$, we have $3\leq diam(T_{i}')\leq 4$ \ $(1\leq i\leq t)$, i.e., each $T_{i}'$ is either a $P_4$ or a $P_2$-star. Then, it follows that $m_{L(T_{i}')}(1)=0$ $(1\leq i\leq t)$ by Lemmas \ref{lem:path} and \ref{lem:star}. Applying Lemma \ref{lem:edge} and \ref{lem:tree} and $n\geq 12$, we obtain
$$m_{L(T)}(1)\leq 1+m_{L(T+u)}(1)=1+\sum_{i=1}^{t} m_{L(T_{i}')}(1)=1<\frac{n-5}{6}.$$

In the following, suppose that $diam(T)\geq 7$.
Since $v_4$ and $v_5$ lie on the diametrical path of $T$, then $d(v_4)\geq 2$ and $d(v_5)\geq 2$. Thus, we proceed by discussing the degrees of $v_4$ and $v_5$:\ $d(v_4)=d(v_5)=2$, or $d(v_4)\geq 3$, or $d(v_4)=2$ and $d(v_5)\geq 3$.

\vskip 2mm
\noindent
{\bf Case 1.} \ Assume that $d(v_4)=d(v_5)=2$. \\
\indent
Denote by $T_{v_5}$ (resp., $T_{v_6}$) the component of $T-e_{v_5 , v_6}$ (deleting the edge $e_{v_5 , v_6}$ from $T$), which contains the vertex $v_5$ (resp., $v_6$). Clearly, $T_{v_5}$ is a $P_2$-star and $m_{L(T_{v_5})}(1)=0$ by Lemma \ref{lem:star}. Then it follows from Lemma \ref{lem:edge} that
\begin{equation}\label{e1}
  m_{L(T)}(1)\leq 1+m_{L(T_{v_5})}(1)+m_{L(T_{v_6})}(1)=1+m_{L(T_{v_6})}(1).
\end{equation}
Note that $|T_{v_6}|=|T|-|T_{v_5}|\leq n-7$ and  $T_{v_6}$ falls into three cases: (a) $T_{v_6}\in \Omega$; (b) $T_{v_6} \notin \Omega$ and $T_{v_6}$ contains no $\mathcal{P}^3$; (c) $T_{v_6}$ contains $\mathcal{P}^3$. Accordingly, we prove this case by the following three subcases.

(a) If $T_{v_6}\in \Omega$, then $|T_{v_6}|\geq 7$.

Applying the induction hypothesis to $T_{v_6}$, we have
$m_{L(T_{v_6})}(1)\leq \frac{|T_{v_6}|-5}{6}.$
Thus it follows from (\ref{e1}) and $|T_{v_6}|\leq n-7$ that
\begin{equation*}
 \begin{array}{rcl}
  m_{L(T)}(1)&\leq & 1+m_{L(T_{v_6})}(1)\\
  &\leq& 1+\frac{|T_{v_6}|-5}{6}\\
  &<& \frac{n-5}{6}.
  \end{array}
\end{equation*}

(b) Suppose that $T_{v_6} \notin \Omega$ and $T_{v_6}$ contains no $\mathcal{P}^3$.

For this subcase, we see that $v_6$ is a pendant vertex and $v_7$ is a quasi-pendant vertex with $d(v_7)\geq 3$ in $T_{v_6}$. Then we use the reduction operation to $T_{v_6}$ with respect to $v_6$ and $v_7$. Let $H_i'\ (1\leq i\leq s)$ be all the components of the reduction graph of $T_{v_6}$. Then each $H_i'$ is a path $P_4$ or a tree of $\Omega$ or a tree containing $\mathcal{P}^3$. If some component $H_i'$ contains a $\mathcal{P}^3$, then by Lemma \ref{lem:path} we turn to $H_i'-\mathcal{P}^3$ instead of $H_i'$. Without loss of generality, assume that $H_i'-\mathcal{P}^3$ contains no $\mathcal{P}^3$, then $H_i'-\mathcal{P}^3$ is a path $P_2$ or a tree of $\Omega$ or a tree containing a quasi-pendant vertex of degree greater than 2. If $H_i'-\mathcal{P}^3$ is a tree containing a quasi-pendant vertex of degree greater than 2, then its structure is similar with $T_{v_6}$ and we could handle it with above process again. At last, it suffices to count the components belonging to $\Omega$. If there is no component belonging to $\Omega$ after doing the process above, then $m_{L(T_{v_6})}(1)=0$, and thus
$$m_{L(T)}(1)\leq 1+m_{L(T_{v_6})}(1)=1<\frac{n-5}{6}.$$
Now suppose that there are $k(\geq 1)$ components, say $\Gamma_i'\ (1\leq i\leq k)$, belonging to $\Omega$. Then by Lemmas \ref{lem:tree} and \ref{lem:path} and the induction hypothesis on $\Gamma_i'$,
\begin{equation*}
 \begin{array}{rcl}
m_{L(T_{v_6})}(1)&=&\sum\limits_{i=1}^{k} m_{L(\Gamma_{i}')}(1)\\
&\leq& \sum\limits_{i=1}^{k} \frac{|\Gamma_{i}'|-5}{6}\\
&\leq& \frac{|T_{v_6}|+2(k-1)-5k}{6}\\
&=&\frac{|T_{v_6}|-3k-2}{6}.
  \end{array}
\end{equation*}
Consequently, it follows by (\ref{e1}) and $|T_{v_6}|\leq n-7$ that
\begin{equation*}
 \begin{array}{rcl}
m_{L(T)}(1)&\leq& 1+m_{L(T_{v_6})}(1)\\
&\leq& 1+\frac{|T_{v_6}|-3k-2}{6}\\
&\leq& \frac{n-3k-3}{6}\\
&\leq& \frac{n-6}{6}\\
&<&\frac{n-5}{6}.
  \end{array}
\end{equation*}

(c) Assume that $T_{v_6}$ contains a $\mathcal{P}^3$.

In this subcase, we say $m_{L(T_{v_6})}(1)=m_{L(T_{v_6}-\mathcal{P}^3)}(1)$ by Lemma \ref{lem:path}. Without loss of generality, let $T_{v_6}-\mathcal{P}^3$ contain no $\mathcal{P}^3$. If $T_{v_6}-\mathcal{P}^3$ is a path $P_2$, then $T_{v_6}$ is a path $P_5$ and $m_{L(T_{v_6})}(1)=0$, and thus by (\ref{e1})
$$m_{L(T)}(1)\leq 1+m_{L(T_{v_6})}(1)=1<\frac{n-5}{6}.$$
Otherwise, $T_{v_6}-\mathcal{P}^3$ returns to the subcase $(a)$ or $(b)$, which indicates that $$m_{L(T)}(1)<\frac{n-5}{6}.$$

\vskip 2mm
\noindent
{\bf Case 2.} \ Assume that $d(v_4)\geq 3$.

In this case, first suppose that $d(v_4)=3$ and there is exactly one $\mathcal{P}^2$ adjacent to $v_4$. Denote by $T_{v_4}$ (resp., $T_{v_5}$) the component of $T-e_{v_4,v_5}$, which contains the vertex $v_4$ (resp., $v_5$).
Then $|T_{v_4}|\geq 8$ and $m_{L(T_{v_4})}(1)=0$ from Lemmas \ref{lem:path} and \ref{lem:star}. The graph $T_{v_5}$ has three cases: $T_{v_5}\in \Omega$; $T_{v_5} \notin \Omega$ and containing no $\mathcal{P}^3$; $T_{v_5}$ containing $\mathcal{P}^3$. Applying parallel discussion with $T_{v_6}$ in Case 1, one can also derive that
$m_{L(T)}(1)<\frac{n-5}{6}.$

Next, denote by $T+u$ the graph obtained from $T$ and a new vertex $u$ by joining $u$ with $v_4$. Applying reduction operation to $T+u$ with respect to $u$ and $v_4$, we let $T_{v_5}'$ be the component of the reduction graph of $T+u$, which contains $v_5$. Then each of other components of the reduction graph of $T+u$ is a path $P_4$ or a $P_2$-star, which contains no Laplacian eigenvalue 1 by Lemma \ref{lem:path} or Lemma \ref{lem:star}. Thus it follows by Lemmas \ref{lem:edge} and \ref{lem:tree} that
$$m_{L(T)}(1)\leq 1+m_{L(T+u)}(1)=1+m_{L(T_{v_5}')}(1).$$
Note that $|T_{v_5}'|\leq n-8$ and $T_{v_5}'\in \Omega$ or $T_{v_5}'$ contains a $\mathcal{P}^3$. Using parallel proof with $T_{v_6}$ in (a) or (c) of Case 1, we also obtain that $m_{L(T)}(1)<\frac{n-5}{6}.$

\vskip 2mm
\noindent
{\bf Case 3.} \ Assume that $d(v_4)=2$ and $d(v_5)\geq 3$.

Denote by $T+u$ the graph obtained by adding a new vertex $u$ adjacent to $v_5$ of $T$. Applying reduction operation on $T+u$ with respect to $u$ and $v_5$, we let $T_{v_6}'$ be the component of the reduction graph of $T+u$ that contains $v_6$. Then $|T_{v_6}'|\leq n+1-8=n-7$ holds. Except for $T_{v_6}'$, each component of the reduction graph of $T+u$ is a path $P_4$, a $P_2$-star  or a tree obtained by joining a pendant vertex of a $P_3$ to the central  vertex of a $P_2$-star, which contains no Laplacian eigenvalue 1 by Lemmas \ref{lem:path} or Lemma \ref{lem:star}. Thus it follows from Lemmas \ref{lem:edge} and \ref{lem:tree} that
$$m_{L(T)}(1)\leq 1+m_{L(T+u)}(1)=1+m_{L(T_{v_6}')}(1).$$
It is not hard to see that $T_{v_6}'$ satisfies the situation (a) or (c) of Case 1 above. Hence, we could also derive that $m_{L(T)}(1)<\frac{n-5}{6}$.

Combining the above process, it is proved that $m_{L(T)}(1)<\frac{n-5}{6}$ when $n\geq 12$. Recalling the beginning discussion with $7\leq n\leq 11$, we say that $m_{L(T)}(1)=\frac{n-5}{6}$ if and only if $n=11$ and $T$ is isomorphic to the graph in Fig. \ref{fig.1:eps}.

As a consequence, the proof of Theorem \ref{thm:tree} is completed. \hfill$\square$

\section{Proof of Theorem 1.2}
\noindent
{\bf Proof} \ \ Let $G$ be a reduced graph with $n(\geq 6)$ vertices and $m$ edges. We prove this theorem by induction on the first Betti number $c=m-n+1$. The vertices on a cycle of $G$ are called cycle-vertices. First, we show that the conclusion holds for $c=0$ and $c=1$.
If $c=0$, then $G$ is a reduced tree. It follows from Lemma  \ref{lem:reduced} that
$$m_{L(G)}(1)\leq \frac{n-2}{4}= c+\frac{n-2}{4},$$
and equality holds if and only if those two assertions hold.

If $c=1$, then $G$ is a reduced unicyclic graph. We proceed by considering whether there exists a quasi-pendant vertex on the unique cycle of $G$. The proof is divided into the following two cases.

\vskip 2mm
\noindent
{\bf Case 1.} \ Assume that there exists a quasi-pendant vertex on the unique cycle of $G$.

Let $u$ be a quasi-pendant vertex on the cycle and let $v$ be another cycle-vertex with $v\sim u$. Applying Lemma \ref{mainlemma} in terms of $u$ and $v$, we denote by $\Gamma$ the resultant graph. Then $\Gamma$ is a reduced tree of order $n+2$. Thus Lemmas \ref{mainlemma} and \ref{lem:reduced} imply that
$$m_{L(G)}(1)=m_{L(\Gamma)}(1)\leq \frac{(n+2)-2}{4}=\frac{n}{4}<1+\frac{n-2}{4}.$$

\vskip 2mm
\noindent
{\bf Case 2.} \ None of the cycle-vertices is a quasi-pendant vertex.

First, suppose that $G$ is the cycle $C_{n}$. It is known that the Laplacian eigenvalues of $C_n$ are $2-2\cos \frac{2k\pi}{n} \ (k=0,1,\,\cdots,\,n-1)$. Then $m_{L(C_{n})}(1)=2$, if $n\equiv 0(mod\ 6)$; otherwise, $m_{L(C_{n})}(1)=0$. Thus we obtain
\begin{equation*}
	m_{L(C_{n})}(1)\leq 1+\frac{n-2}{4},
\end{equation*}
and equality holds if and only if $G=C_6$.

Next, suppose $G$ is not a cycle. Then there is a cycle-vertex, say $x$, with $d(x)> 2$ in $G$.
On the one hand, if there exists an edge $e$ on the cycle of $G$ such that $x$ is not a quasi-pendant vertex in $G-e$ with degree also greater than 2, then $G-e$ is a reduced tree but not satisfying the assertion (i) of Lemma \ref{lem:reduced}. Thus it follows from Lemmas \ref{lem:edge} and \ref{lem:reduced} that
$$m_{L(G)}(1)\leq 1+m_{L(G-e)}(1)<1+\frac{n-2}{4}.$$
On the other hand, if there is no that kind of edge on the cycle of $G$, then we just need to consider the following situations.

(i) \ The length of the cycle is 3 in $G$ and there is exactly one cycle-vertex with degree 3 (i.e., $G$ has a local structure depicted as $G_1$ in Fig. \ref{fig.5:eps}).

\begin{figure}[htbp]
	\centering
	\includegraphics[width=0.8\textwidth]{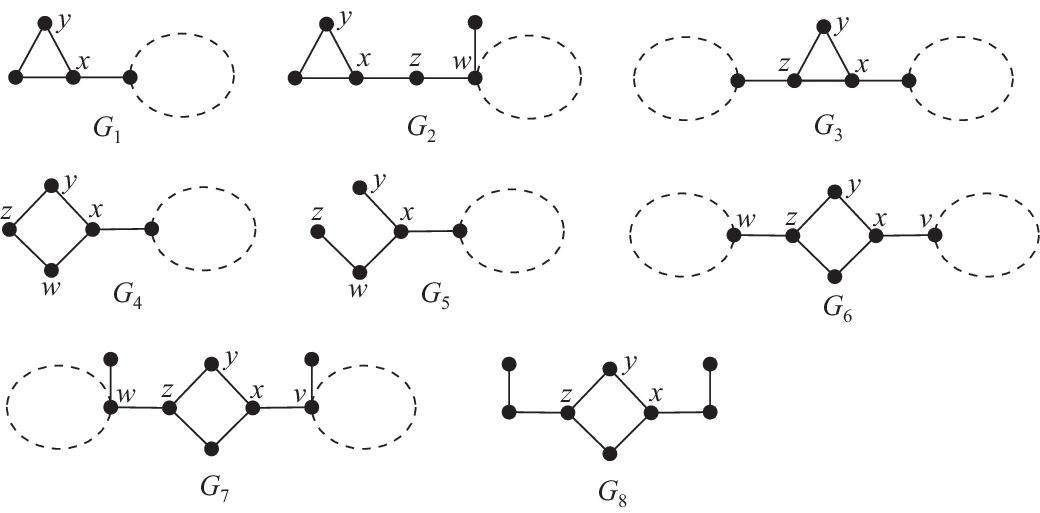}
	\caption{The graphs $G_1\thicksim G_6$ with the dashed cycle meaning the remaining part.}
	\label{fig.5:eps}
\end{figure}

For this situation, we consider the graph $G_1-e_{xy}$, which is a reduced tree clearly. Then by Lemmas \ref{lem:edge} and \ref{lem:reduced}, we get
\begin{equation}\label{e2}
m_{L(G_1)}(1)\leq 1+m_{L(G_1-e_{xy})}(1)\leq 1+\frac{n-2}{4}.
\end{equation}
Actually, $m_{L(G_1)}(1)$ cannot attain the upper bound. If $|G_1|=6$, then it is not hard to see that $m_{L(G_1)}(1)=0<1+\frac{n-2}{4}=2$. Now, let $|G_1|\geq 7$ and suppose for a contradiction that $m_{L(G_1)}(1)=1+\frac{n-2}{4}$. Then it follows from (\ref{e2}) that
\begin{equation*}
	m_{L(G_1-e_{xy})}(1)=\frac{n-2}{4}.
\end{equation*}
Thus, $G_1-e_{xy}$ is the extremal graph characterized in Lemma \ref{lem:reduced}, each  final reduction component of which is a path $P_6$. That is to say, $G$ must have a local structure depicted as $G_2$ in Fig. \ref{fig.5:eps}. Denote by $G_w'$ and $G_z'$ the two components obtained by applying Lemma \ref{mainlemma} to $G_2$ in terms of $w$ and $z$. Clearly, $m_{L(G_z')}(1)=0$ and $G_w'$ is a reduced tree. Hence, by Lemmas \ref{mainlemma} and \ref{lem:reduced},
\begin{equation*}
 \begin{array}{rcl}
m_{L(G_2)}(1)&=&m_{L(G_z')}(1)+m_{L(G_w')}(1)\\
&=& m_{L(G_w')}(1)\leq \frac{|G_w'|-2}{4}\\
&=& \frac{(n-4)-2}{4}<1+\frac{n-2}{4},
  \end{array}
\end{equation*}
contradicting with $m_{L(G_1)}(1)=1+\frac{n-2}{4}$. As a result, for this situation we say that $$m_{L(G)}(1)<1+\frac{n-2}{4}.$$

(ii) \ $G$ has a local structure depicted as $G_3$ in Fig. \ref{fig.5:eps}.

For graph $G_3$, we consider $G_3-e_{xz}$, a reduced tree, to prove $m_{L(G_3)}(1)<1+\frac{n-2}{4}$. Since $d(x)=d(y)=d(z)=2$ in $G_3-e_{xz}$, then we say that the component containing $\{x, y, z\}$ in the final reduction components of $G_3-e_{xz}$ cannot be a path $P_6$. Thus by Lemma \ref{lem:reduced},
$m_{L(G_3-e_{xz})}(1)< 1+\frac{n-2}{4}.$
It follows from Lemma \ref{lem:edge} that
\begin{equation*}
m_{L(G_3)}(1)\leq 1+m_{L(G_3-e_{xz})}(1)< 1+\frac{n-2}{4}.
\end{equation*}

(iii) \ $G$ has a local structure depicted as $G_4$ in Fig. \ref{fig.5:eps}.

For graph $G_4$, we delete the edge $e_{yz}$ from $G_4$, obtaining the graph $G_5$ (see Fig. \ref{fig.5:eps}). Applying Lemma \ref{mainlemma} to $G_5$ in terms of $x$ and $w$, we obtain two components: a path $P_4$ (containing $w, z$) and a reduced tree, written as $T_x'$ (containing $x, y$). Then by Lemma \ref{mainlemma},
\begin{equation}\label{e3}
m_{L(G_5)}(1)=m_{L(P_4)}(1)+m_{L(T_x')}(1)=m_{L(T_x')}(1).
\end{equation}
If $|T_x'|=4$ or $5$ (i.e., $T_x'$ is a $P_4$ or $P_5$), then $m_{L(G_5)}(1)=0$ from (\ref{e3}), and thus by Lemma \ref{lem:edge} $$m_{L(G_4)}(1)\leq 1+m_{L(G_5)}(1)=1<1+\frac{n-2}{4}.$$
If $|T_x'|\geq 6$, then by Lemmas \ref{lem:edge} and \ref{lem:reduced} and (\ref{e3})
\begin{equation*}
 \begin{array}{rcl}
m_{L(G_4)}(1)&\leq&  1+m_{L(G_5)}(1)=1+m_{L(T_x')}(1)\\
&\leq& 1+ \frac{|T_x'|-2}{4}=1+\frac{(n-2)-2}{4}\\
&<& 1+\frac{n-2}{4}.
  \end{array}
\end{equation*}

(iv) \ $G$ has a local structure depicted as $G_6$ in Fig. \ref{fig.5:eps}.

First, considering the reduced tree $G_6-e_{xy}$, we obtain from Lemmas \ref{lem:edge} and \ref{lem:reduced} that \begin{equation}\label{e4}
  m_{L(G_6)}(1)\leq 1+m_{L(G_6-e_{xy})}(1)\leq 1+\frac{n-2}{4}.
\end{equation}
Suppose that $m_{L(G_6)}(1)=1+\frac{n-2}{4}$, then $m_{L(G_6-e_{xy})}(1)=\frac{n-2}{4}$ by (\ref{e4}), and thus $G_6-e_{xy}$ satisfies those two assertions in Lemma \ref{lem:reduced}. To guarantee each final reduction component of $G_6-e_{xy}$ is $P_6$, we say that the vertex $v$ in $G_6$ is a quasi-pendant vertex.

Next, Considering the reduced tree $G_6-e_{zy}$ by parallel discussion with above, one could obtain that the vertex $w$ in $G_6$ is also a quasi-pendant vertex. That is to say, $G$ has a local structure depicted as $G_7$ in Fig. \ref{fig.5:eps}. If $|G_7|=8$, then $G$ is isomorphic to $G_8$ in Fig. \ref{fig.5:eps} and $m_{L(G_8)}(1)=0<1+\frac{n-2}{4}$ by calculation. If $|G_7|>8$, we apply Lemma \ref{mainlemma} to $G_7$ in terms of $\{w,z\}$ and $\{x,v\}$ respectively, and denote the resultant graph by $G_7'$. Then $G_8$ is one component of $G_7'$ with $m_{L(G_8)}(1)=0$.
Without loss of generality, let the other two components (say $H_1$ and $H_2$) of $G_7'$ be of order not less than 6 (otherwise they are paths $P_2$, $P_4$ or $P_5$). Hence by Lemmas \ref{mainlemma} and \ref{lem:reduced},
\begin{equation*}
  \begin{array}{rcl}
 m_{L(G_7)}(1)&=&m_{L(G_7')}(1)=m_{L(G_8)}(1)+m_{L(H_1)}(1)+m_{L(H_2)}(1)\\
 &=& m_{L(H_1)}(1)+m_{L(H_2)}(1)\\
 &\leq& \frac{|H_1|-2}{4}+\frac{|H_2|-2}{4}\\
 &=& \frac{n-8}{4}<1+\frac{n-2}{4}.
  \end{array}
\end{equation*}

Combining the above discussions, it is proved for a reduced unicyclic graph $G$ on $n\geq 6$ vertices that $m_{L(G)}(1)\leq 1+\frac{n-2}{4}$, and $m_{L(G)}(1)= 1+\frac{n-2}{4}$ if and only if $G$ is $C_6$. In other words, if $n\geq 7$, then $m_{L(G)}(1)< 1+\frac{n-2}{4}$ for a reduced unicyclic graph $G$.

Now suppose that $G$ contains at least two cycles, i.e., $c\geq 2$. If $n=6$, we check directly by computer and obtain that
$$m_{L(G)}(1)<c+\frac{n-2}{4}.$$
For $n\geq 7$, we assume that $m_{L(G)}(1)<c+\frac{n-2}{4}$ holds for any reduced graph $G$ with cyclomatic number $c-1$. On the one hand, if there is a cycle-vertex (say $u$) that is a quasi-pendant vertex in $G$, then let $v\thicksim u$ and $v$ be also a cycle-vertex and quasi-pendant vertex. Denote by $G'$ the graph obtained by applying Lemma \ref{mainlemma} to $G$ in terms of $u$ and $v$. Then $G'$ is a reduced graph with cyclomatic number $c-1$. Thus, by Lemma \ref{mainlemma} and induction hypothesis on $G'$,
\begin{equation*}
  \begin{array}{rcl}
m_{L(G)}(1)&=&m_{L(G')}(1)<c-1+\frac{|G'|-2}{4}\\
&=&c-1+\frac{(n+2)-2}{4}= c+\frac{n-4}{4}\\
&<& c+\frac{n-2}{4}.
  \end{array}
\end{equation*}
On the other hand, if every cycle-vertex is not a quasi-pendant vertex in $G$, then let $e$ be an edge on one cycle of $G$. Note that $G-e$ is a reduced graph with cyclomatic number $c-1$. Hence, by Lemma \ref{lem:edge} and induction hypothesis on $G-e$,
$$m_{L(G)}(1)\leq 1+m_{L(G-e)}(1)<1+c-1+\frac{|G-e|-2}{4}=c+\frac{n-2}{4}.$$

Consequently, we prove that $m_{L(G)}(1)\leq c+\frac{n-2}{4}$ for a reduced graph on $n\geq 6$ vertices, and  $m_{L(G)}(1)= c+\frac{n-2}{4}$ if and only if $c=0$ and $G$ is a reduced tree satisfying those two assertions in Lemma \ref{lem:reduced}, or $c=1$ and $G=C_6$.
The proof of Theorem \ref{thm:circle} is completed. \hfill$\square$


{\small
\begin{thebibliography}{90}
\bibitem{Akbari} S. Akbari, E.R. van Dam, M.H. Fakharan, Trees with a large Laplacian  eigenvalue multiplicity, Linear Algebra Appl. 586 (2020) 262–273.

\bibitem{Andrade} E. Andrade, D.M. Cardoso, G. Pastén, O. Rojo, On the Faria’s  inequality for the Laplacian and signless Laplacian spectra: a unified approach, Linear Algebra Appl. 472 (2015) 81–96.

\bibitem{Barik} S. Barik, A.K. Lal, S. Pati, On trees with Laplacian eigenvalue one, Linear Multilinear Algebra 56 (6) (2008) 597–610.

\bibitem{Du} Z. Du, The multiplicity of eigenvalues of $A_{\alpha}$ of trees, Ann. Comb. (2025), https://doi.org/10.1007/s00026-025-00796-5.

\bibitem{Faria} I. Faria, Permanental roots and the star degree of a graph, Linear Algebra Appl. 64 (1985) 255–265.

\bibitem{Guo} J.-M. Guo, L. Feng, J. Zhang, On the multiplicity of Laplacian eigenvalues of graphs, Czechoslov. Math. J. 60 (135) (2010) 689–698.

\bibitem{Gutman} I. Gutman, I. Sciriha, On the nullity of line graphs of trees, Discrete Math. 232 (2001) 35–45.

\bibitem{Grone} R. Grone, R. Merris, V.S. Sunder, The Laplacian spectrum of a graph, SIAM J. Matrix Anal. Appl. 11 (1990) 218–238.

\bibitem{Tian1} F. Tian, J. Wang, W. Song, Upper bound of the multiplicity of Laplacian eigenvalue 1 of trees, Linear Algebra Appl. 724 (2025) 108–119.

\bibitem{Tian} F. Tian, D. Wong, On the multiplicity of 1 as a Laplacian eigenvalue of a graph, Discrete Math. 349 (2026) 115030.

\bibitem{WangZ} Z. Wang, Q. Chen, J. Guo, X. Li. A relation between multiplicity of 1 as a Laplacian eigenvalue and induced matching numbers in trees. Discrete Math. 348(5) (2025) 114401.

\bibitem{DWong} X. Wang, Q. Zhou, D. Wong, C. Zhang, F. Tian, The multiplicity of laplacian eigenvalue two in a connected graph with a perfect matching. Linear Algebra Appl. 609 (2021) 152-162.

\bibitem{Wen} F. Wen, Q. Huang, On the multiplicity of Laplacian eigenvalues for unicyclic graphs. Czech. Math. J. 72 (147) (2022) 371-390.

\bibitem{WongZX} D. Wong, W. Zhen, S. Xu, Characterization of trees with Laplacian eigenvalue multiplicity one less than the number of pendant vertices. Discrete Math. 349 (2026) 114799.

\bibitem{Yang} J. Yang, L. Wang, Line graphs of trees with the largest eigenvalue multiplicity, Linear Algebra Appl. 676 (2023) 56–65.
\end{thebibliography}
}
\end{document}